\RequirePackage{fix-cm}
\documentclass[smallextended]{svjour3}       
\usepackage{amssymb}
\smartqed  
\usepackage{graphicx}
\begin{document}

\title{SOME PROPERTIES OF BOUNDED TRI-LINEAR MAPS
}
\subtitle{Do you have a subtitle?\\ If so, write it here}


\author{Abotaleb Sheikhali \and Ali Ebadian\and Kazem Haghnejad Azar 
}


\institute{A. Sheikhali \at
             Department of Mathematics, Payame Noor University (PNU), Tehran, Iran.\\
              \email{Abotaleb.sheikhali.20@gmail.com}           
                                   \and
           A. Ebadian \at
              \email{Ebadian.ali@gmail.com}
               \and
           K. Haghnejad Azar \at
              \email{Haghnejad@aut.ac.ir}
}

\date{Received: date / Accepted: date}

\maketitle

\begin{abstract}
Let $X,Y,Z$ and $W$ be normed spaces and $f:X\times Y\times Z\longrightarrow W $ be a bounded tri-linear mapping. In this Article, we define the topological centers for bounded tri-linear mapping and we invistagate thier properties. We study the relationships between weakly compactenss of bounded linear mappings and regularity of bounded tri-linear mappings. For both bounded tri-linear mappings   $f$ and $g$, let $f$ factors through $g$, we present necessary and suficient condition such that the extensions of $f$ factors through extensions of $g$. Also we establish relations between  regularity and factorization  property of bounded tri-linear mappings.
\keywords{Arens product\and Module action\and Factors \and Topological center\and Tri-linear mappings.}
 \subclass{MSC 46H25 \and MSC 46H20\and MSC 46L06}
\end{abstract}

\section{Introduction}
\label{intro}
Let $X, Y, Z$ and $W$ be normed spaces and $f:X\times Y\times Z\longrightarrow W $ be a bounded tri-linear mapping. The natural extensions  of $f$ are  as following:
\begin{enumerate}
\item $f^{*}:W^{*}\times X\times Y\longrightarrow Z^{*}$, given by $\langle f^{*}(w^{*},x,y),z\rangle=\langle w^{*},f(x,y,z)\rangle$ where $x\in X, y\in Y, z\in Z, w^{*}\in W^{*}$. 

The map $f^*$ is a bounded tri-linear mapping and is said  the adjoint of $f$.

\item $f^{**}=(f^*)^*:Z^{**}\times W^{*}\times X\longrightarrow Y^{*}$, given by $\langle f^{**}(z^{**},w^{*},x),y\rangle=\langle z^{**}, f^{*}(w^{*},x,y)$ where $x\in X, y\in Y, z^{**}\in Z^{**}, w^{*}\in W^{*}$.

\item $f^{***}=(f^{**})^*:Y^{**}\times Z^{**}\times W^{*}\longrightarrow X^{*}$, given by $\langle f^{***}(y^{**},z^{**},w^{*}),x\rangle=\langle y^{**},  f^{**}(z^{**},w^{*},x) \rangle$ where $x\in X, y^{**}\in Y^{**}, z^{**}\in Z^{**}, w^{*}\in W^{*}$.

\item $f^{****}=(f^{***})^*:X^{**}\times Y^{**}\times Z^{**}\longrightarrow W^{**}$, given by $\langle f^{****}(x^{**},
y^{**},z^{**})$, $w^{*}\rangle =\langle x^{**}, f^{***}(y^{**},z^{**},w^{*}) \rangle$ where $x^{**}\in X^{**}, y^{**}\in Y^{**}, z^{**}\in Z^{**}, w^{*}\in W^{*}$.
\end{enumerate}
Now let $f^r:Z\times Y\times X\longrightarrow W$ be the flip of $f$ defined by $f^r(z,y,x)=f(x,y,z)$, for every $x\in X ,y\in Y$ and $z\in Z$. Then $f^r$ is a bounded tri-linear map and  it may  extends as above to $f^{r****}:Z^{**}\times Y^{**}\times X^{**}\longrightarrow W^{**}$. When $f^{****}$ and $f^{r****r}$ are equal, then $f$  is called  regular. 
Regularity of $f$ is equvalent to the following
$$w^{*}-\lim\limits_\alpha w^{*}-\lim\limits_\beta w^{*}-\lim\limits_\gamma f(x_\alpha,y_\beta,z_{\gamma})=w^{*}-\lim\limits_\gamma w^{*}-\lim\limits_\beta w^{*}-\lim\limits_\alpha f(x_\alpha,y_\beta,z_{\gamma}),$$
where $\{x_{\alpha} \}, \{y_{\beta} \}$ and $\{z_{\gamma} \}$ are nets in $X, Y$ and $Z$  which converge to $x^{**}\in X^{**},y^{**}\in Y^{**}$ and $z^{**}\in Z^{**}$  in the $w^{*}-$topologies, respectively. For a bounded tri-linear map $f:X\times Y\times Z\longrightarrow W$, if from $X, Y$ or $Z$  at least two reflexive then f is regular. 

A bounded bilinear(res tri-linear) mapping $m:X\times Y\longrightarrow Z$(res $f:X\times Y\times Z\longrightarrow W $) is said to be factor if it is surjective, that is $f(X\times Y)= Z$(res $ f(X\times Y\times Z) =W) $. For a good source of information on this subject, we refer the reader to \cite{5}. 

For a discussion of Arens regularity for Banach algebras and bounded bilinear maps, see \cite{1}, \cite{2}, \cite{11}, \cite{12} and \cite{18}. For example, every $C^{*}-$algebra is Arens regular, see \cite{4}. Also $L^{1}(G)$ is Arens regular if and only if G is finite,\cite{19}.\\
The left topological center of $m$ may be defined as following

$Z_{l}(m) = \{x^{**} \in X^{**} : y^{**} \longrightarrow m^{***}(x^{**}, y^{**}′)\ is\ weak^{*}-to-weak^{*}-continuous\}.$

Also the right topological center of $m$ as

$Z_{r}(m) = \{y^{**} \in Y^{**} : x^{**} \longrightarrow m^{r***r}(x^{**}, y^{**}′)\ is\ weak^{*}-to-weak^{*}-continuous\}.$

The subject of topological centers have been investigated in \cite{6}, \cite{7} and \cite{15}.
In \cite{13}, Lau and Ulger gave several significant results related to the topological centers of certain dual algebras. In \cite{11}, Authors extend some problems from Arens regularity and  Banach algebras to module actions. They also extend the definitions of the left and right multiplier for module actions, see \cite{10} and \cite{12}.

Let $A$ be a Banach algebra, and let $\pi :A\times A\longrightarrow A $ denote the product of A,
so that $\pi (a,b) = ab$ for every $a,b\in A$. The Banach algebra $A$ is Arens regular whenever the map $\pi$ is Arens
regular. The first and second Arens products denoted by $\square$ and $\lozenge$ respectively and definded by 

$a^{**}\square b^{**}=\pi^{***}(a^{**},b^{**})\ \ ,\ \ a^{**}\lozenge b^{**}=\pi^{r***r}(a^{**},b^{**})\ \ \ ,\ \ \ (a^{**},b^{**}\in A^{**}).$

\section{\textbf{Module actions for bounded tri-linear maps}}
\label{sec:2}
Let $(\pi_{1},X,\pi_{2})$ be a Banach $A-$module and let $\pi_{1}:A\times X \longrightarrow X $ and $\pi_{2}:X\times A \longrightarrow X$ be the left and right module actions of $A$ on $X$, respectively. If $(\pi_{1},X)$ (res $(X,\pi_{2})$)  is a left (res right) Banach $A-$module of $A$ on $X$, then $(X^{*},\pi_{1}^{*})$(res $(\pi_{2}^{r*r},X^{*})$) is a right (res left) Banach $A-$module and $(\pi_{2}^{r*r},X^{*},\pi_{1}^{*})$ is the dual Banach $A-$module of $(\pi_{1},X,\pi_{2})$. We note also that $(\pi_{1}^{***},X^{**},\pi_{2}^{***})$ is a Banach $(A^{**},\square)-$module with module actions $\pi_{1}^{***}:A^{**}\times X^{**} \longrightarrow X^{**}$ and $\pi_{2}^{***}:X^{**}\times A^{**} \longrightarrow X^{**}$. Similary, $(\pi_{1}^{r***r},X^{**},\pi_{2}^{r***r})$ is a Banach $(A^{**},\lozenge)-$module with module actions $\pi_{1}^{r***r}:A^{**}\times X^{**} \longrightarrow X^{**}$ and $\pi_{2}^{r***r}:X^{**}\times A^{**} \longrightarrow X^{**}$.
If we continue dualizing we shall reach $(\pi_{2}^{***r*r},X^{***},\pi_{1}^{****})$ and $(\pi_{2}^{r****r},X^{***},\pi_{1}^{r***r*})$ are the dual Banach $(A^{**},\square)$
$-$ module and Banach $(A^{**},\lozenge)-$module of $(\pi_{1}^{***},X^{**}$
$,\pi_{2}^{***})$ and $(\pi_{1}^{r***r}$, $X^{**},\pi_{2}^{r***r})$, respectively. In \cite{8}, Eshaghi Gordji and Fillali show that  if  a Banach algebra $A$  has a bounded left (or right) approximate identity, then the left (or right) module action of $A$ on $A^{*}$ is Arens regular if and only if $A$ is reflexive. 

We commence with the following definition for bounded tri-linear mapping.
\begin{definition}\label{2.1}
Let $X$  be a Banach space, $A$  be a Banach algebra and let $\Omega_{1}:A\times A\times X \longrightarrow X$ be a bounded tri-linear map. Then the pair $(\Omega_{1},X)$ is said to be a left Banach $A-$module when $$\Omega_{1}(\pi(a,b),\pi(c,d),x)=\Omega_{1}(a,b,\Omega_{1}(c,d,x)),$$ for each $a,b,c,d\in A$ and $x\in X$.  A right Banach $A-$module may be defined similarly. Let  $\Omega_{2}:X\times A\times A  \longrightarrow X$ be a bounded tri-linear map. Then the pair $(X,\Omega_{2})$ is said to be a right Banach $A-$module when $$\Omega_{2}(x,\pi(a,b),\pi(c,d))=\Omega_{2}(\Omega_{2}(x,a,b),c,d).$$ A triple $(\Omega_{1},X,\Omega_{2})$ is said to be a Banach $A-$module when  $(\Omega_{1},X)$ and $(X,\Omega_{2})$ are left and right Banach $A-$modules respectively, also $$\Omega_{2}(\Omega_1(a,b,x),c,d)=\Omega_{1}(a,b,\Omega_{2}(x,c,d)).$$
\end{definition}
If $(\Omega_{1},X,\Omega_{2})$ is a Banach $A-$module, then $(\Omega_{2}^{r*r},X^{*},\Omega_{1}^{*})$ is a Banach $A-$module.
It follows that, 

\begin{enumerate}
\item the triple $(\Omega_{1}^{****},X^{**},\Omega_{2}^{****})$ is a Banach $(A^{**},\square,\square)-$module. 

\item the triple $(\Omega_{1}^{r****r},X^{**},\Omega_{2}^{r****r})$ is a Banach $(A^{**},\lozenge,\lozenge)-$module.
\end{enumerate}

\begin{theorem}\label{2.4}
Let $a,b,c,d\in A$, $x^{*}\in X^{*}$, $x^{**}\in X^{**}$ and $b^{**},c^{**}\in A^{**}$.Then
\begin{enumerate}
\item If $(\Omega_{1},X)$ is a left Banach $A-$module, then
 $$\Omega_{1}^{***}(b^{**},\Omega_{1}^{****}(c,d,x^{**}),x^{*})=\pi^{**}(b^{**},\Omega_{1}^{***}(\pi^{***}(c,d),x^{**},x^{*})),$$
\item If $(X,\Omega_{2})$ is a right Banach $A-$module, then 
$$\Omega_{2}^{r***r}(x^{*},\Omega_{2}^{r****r}(x^{**},a,b),c^{**})=\pi^{r**}(c^{**},\Omega_{2}^{r***r}(x^{*},x^{**},\pi^{***}(a,b)).$$
\end{enumerate}
\end{theorem}
\begin{proof}
(1) Since the pair $(\Omega_{1},X)$ is a  left Banach $A-$module, thus for every $x\in X$ we have 
\begin{eqnarray*}
\langle\Omega_{1}^{*}(x^{*},&\pi&(a,b),\pi(c,d)),x \rangle =\langle x^{*},\Omega_{1}(\pi(a,b),\pi(c,d),x)\rangle \\
&=&\langle x^{*},\Omega_{1}(a,b,\Omega_{1}(c,d,x))\rangle=\langle\Omega_{1}^{*}(x^{*},a,b),\Omega_{1}(c,d,x) \rangle \\
&=&\langle \Omega_{1}^{*}(\Omega_{1}^{*}(x^{*},a,b),c,d),x\rangle.
\end{eqnarray*}
Hence $\Omega_{1}^{*}(x^{*},\pi(a,b),\pi(c,d))=\Omega_{1}^{*}(\Omega_{1}^{*}(x^{*},a,b),c,d)$ and this implies that 
\begin{eqnarray*}
\langle &\pi^{*}&(\Omega_{1}^{***}(\pi^{***}(c,d),x^{**},x^{*}),a),b\rangle=\langle \Omega_{1}^{***}(\pi^{***}(c,d),x^{**},x^{*}),\pi(a,b)\rangle \\
&=&\langle \pi^{***}(c,d),\Omega_{1}^{**}(x^{**},x^{*},\pi(a,b))\rangle=\langle c,\pi^{**}(d,\Omega_{1}^{**}(x^{**},x^{*},\pi(a,b)))\rangle \\
&=&\langle d,\pi^{*}(\Omega_{1}^{**}(x^{**},x^{*},\pi(a,b)),c)\rangle=\langle \Omega_{1}^{**}(x^{**},x^{*},\pi(a,b)),\pi(c,d)\rangle \\
&=&\langle x^{**},\Omega_{1}^{*}(x^{*},\pi(a,b),\pi(c,d))\rangle=\langle x^{**},\Omega_{1}^{*}(\Omega_{1}^{*}(x^{*},a,b),c,d)\rangle \\
&=&\langle \Omega_{1}^{**}(x^{**},\Omega_{1}^{*}(x^{*},a,b),c),d\rangle=\langle \Omega_{1}^{***}(d,x^{**},\Omega_{1}^{*}(x^{*},a,b)),c\rangle \\
&=&\langle \Omega_{1}^{***}(c,d,x^{**}),\Omega_{1}^{*}(x^{*},a,b)\rangle=\langle \Omega_{1}^{**}( \Omega_{1}^{***}(c,d,x^{**}),x^{*},a),b\rangle.
\end{eqnarray*}
Thus $\pi^{*}(\Omega_{1}^{***}(\pi^{***}(c,d),x^{**},x^{*}),a)=\Omega_{1}^{**}( \Omega_{1}^{***}(c,d,x^{**}),x^{*},a)$. Finally, we have 
\begin{eqnarray*}
\langle \Omega_{1}^{***}(b^{**},\Omega_{1}^{****}(c,d,x^{**}),x^{*}),a \rangle &=&\langle b^{**},\Omega_{1}^{**}(\Omega_{1}^{****}(c,d,x^{**}),x^{*},a\rangle \\
&=&\langle b^{**},\pi^{*}(\Omega_{1}^{***}(\pi^{***}(c,d),x^{**},x^{*}),a)\rangle \\
&=&\langle \pi^{**}(b^{**},\Omega_{1}^{***}(\pi^{***}(c,d),x^{**},x^{*})),a\rangle.
\end{eqnarray*}
The proof is complete. A similar argument applies for (2).
\end{proof}

\section{\textbf{Topological centers of bounded tri-linear maps}}
\label{sec:3}

In this section, we shall investigate the topological centers of bounded tri-linear maps.
The main definition of this section is in the following.
\begin{definition}\label{3.1}
Let $f:X\times Y\times Z\longrightarrow W$ be a bounded tri-linear map. We define the topological centers of $f$ by

$Z_{l}^{1}(f)=\{x^{**}\in X^{**}|~y^{**}\longrightarrow f^{****}(x^{**},y^{**},z^{**}) \ is \ weak^{*}-to-weak^{*}- continuous\},$

$Z_{l}^{2}(f)=\{x^{**}\in X^{**}|~z^{**}\longrightarrow f^{****}(x^{**},y^{**},z^{**}) \ is \ weak^{*}-to-weak^{*}- continuous\},$

$Z_{r}^{1}(f)=\{z^{**}\in Z^{**}|~x^{**}\longrightarrow f^{r****r}(x^{**},y^{**},z^{**}) \ is \ weak^{*}-to-weak^{*}- continuous\},$

$Z_{r}^{2}(f)=\{z^{**}\in Z^{**}|~y^{**}\longrightarrow f^{r****r}(x^{**},y^{**},z^{**}) \ is \ weak^{*}-to-weak^{*}- continuous\},$

$Z_{c}^{1}(f)=\{y^{**}\in Y^{**}|~x^{**}\longrightarrow f^{r****r}(x^{**},y^{**},z^{**}) \ is \ weak^{*}-to-weak^{*}- continuous\}.$

$Z_{c}^{2}(f)=\{y^{**}\in Y^{**}|~z^{**}\longrightarrow f^{****}(x^{**},y^{**},z^{**}) \ is \ weak^{*}-to-weak^{*}- continuous\}.$
\end{definition}

For a bounded tri-linear map $f:X\times Y\times Z\longrightarrow W$, We have   
\begin{enumerate}
\item The map $f^{****}$ is the unique extension of $f$ such that $x^{**}\longrightarrow f^{****}(x^{**},y^{**}$

$,z^{**})$ is weak$^{*}-$weak$^{*}$ continuous for each $y^{**}\in Y^{**}$ and $z^{**}\in Z^{**}$.
\item The map $f^{****}$ is the unique extension of $f$ such that $y^{**}\longrightarrow f^{****}(x,y^{**}$

$,z^{**})$ is weak$^{*}-$weak$^{*}$ continuous for each $x\in X$ and $z^{**}\in Z^{**}$.
\item The map $f^{****}$ is the unique extension of $f$ such that $z^{**}\longrightarrow f^{****}(x,y,z^{**})$ is weak$^{*}-$weak$^{*}$ continuous for each $x\in X$ and $y\in Y$.

\item The map $f^{r****r}$ is the unique extension of $f$ such that $z^{**}\longrightarrow f^{r****r}(x^{**}$

$,y^{**},z^{**})$ is weak$^{*}-$weak$^{*}$ continuous for each $x^{**}\in X^{**}$ and $y^{**}\in Y^{**}$.
\item  The map $f^{r****r}$ is the unique extension of $f$ such that $x^{**}\longrightarrow f^{r****r}(x^{**}$

$,y,z)$ is weak$^{*}-$weak$^{*}$ continuous for each $y\in Y$ and $z\in Z$.
\item The map $f^{r****r}$ is the unique extension of $f$ such that $y^{**}\longrightarrow f^{r****r}(x^{**}$

$,y^{**},z)$ is weak$^{*}-$weak$^{*}$ continuous for each $x^{**}\in X^{**}$ and $z\in Z$.
\end{enumerate}
As immediate consequences, we give the next Theorem.
\begin{theorem}\label{3.3}
If $f:X\times Y\times Z\longrightarrow W$ is a bounded tri-linear map, then $X\subseteq Z_{l}^{1}(f)$ and $Z\subseteq Z_{r}^{2}(f)$.
\end{theorem}
The mapping $f^{****}$ is the unique extension of $f$ such that $x^{**}\longrightarrow f^{****}(x^{**},y^{**}$
$,z^{**})$ from $X^{**}$ into $W^{**}$ is weak$^{*}-$ to $-$ weak$^{*}$ continuous for every $y^{**}\in Y^{**}$ and $z^{**}\in Z^{**}$, hence the first right topological center of $f$ may be defined as following

$Z_{r}^{1}(f)=\{z^{**}\in Z^{**}| f^{r****r}(x^{**},y^{**},z^{**})=f^{****}(x^{**},y^{**},z^{**}),\ for\ every\ $

$x^{**}\in X^{**},y^{**}\in Y^{**}\}.$

The mapping $f^{r****r}$ is the unique extension of $f$ such that $z^{**}\longrightarrow f^{r****r}(x^{**},y^{**},z^{**})$ from $Z^{**}$ into $W^{**}$ is weak$^{*}-$ to $-$ weak$^{*}$ continuous for every $x^{**}\in X^{**}$ and $y^{**}\in Y^{**}$, hence the second left topological center of $f$ may be defined as following

$Z_{l}^{2}(f)=\{x^{**}\in X^{**}| f^{r****r}(x^{**},y^{**},z^{**})=f^{****}(x^{**},y^{**},z^{**}),\ for\ every\ $

$ y^{**}\in Y^{**},z^{**}\in Z^{**}\}. $

It is clear that $Z_{r}^{1}(f)=Z_{l}^{2}(f^{r})$ and $Z_{r}^{1}(f^{r})=Z_{l}^{2}(f)$. Also $f$ is regular if and only if $Z_{r}^{1}(f)=Z^{**}$ or $Z_{l}^{2}(f)=X^{**}$. Let $g:X\times X\times X\longrightarrow X$ be a bounded tri-linear map, If $g$ is regular, Then $Z_{r}^{1}(g)=Z_{l}^{2}(g)$.
\begin{example}\label{3.4}
Let $G$ be a finite locally compact Hausdorff group. Then we have
$$f:L^{1}(G)\times L^{1}(G)\times L^{1}(G)\longrightarrow L^{1}(G)$$
defined by $f(k,g,h)=k*g*h$, is regular for every $k, g$ and $h\in L^{1}(G)$. So $Z_{r}^{1}(f)=Z_{l}^{2}(f)=L^{1}(G)$. 
\end{example}
\begin{definition}\label{3.5}
Let $f:X\times Y\times Z\longrightarrow W$ be a bounded tri-linear map. Then
the map $f$ is said to be first left (right) strongly irregular when $Z_{l}^{1}(f)\subseteq X$ ($Z_{r}^{1}(f)\subseteq Z$). The definition of second and third left (right) strongly irregular are similar.
\end{definition}
The proof of the following theorem is straightforward and we left its proof.
\begin{theorem}\label{3.6}
Let $Y$ be a reflexive space and let $f:X\times Y\times Z\longrightarrow W$ be a bounded tri-linear map. Then
\begin{enumerate}
\item The map $f$ is regular and  first right strongly irregular if and only if $Z$  is reflexive.
\item The map $f$ is regular and  second left strongly irregular if and only if $X$  is reflexive.
\end{enumerate}
\end{theorem}

 As immediate consequences of the Theorem \ref{3.6} we have the next corollary.
\begin{corollary}\label{3.7}
Let $f:X\times Y\times Z\longrightarrow W$ be a bounded tri-linear map. If $X$ and $Y$ (or $Z$ and $Y$) are reflexive spaces then $f$ is regular.
\end{corollary}
\begin{corollary}\label{3.8}
Let $A$ be a Banach algebra. If $A$ be reflexive, then 
\begin{enumerate}
\item The bounded tri-linear map $\Omega_{1}$ is regular, first and second left strongly irregular.
\item The bounded tri-linear map $\Omega_{2}$ is regular, first and second right strongly irregular.
\end{enumerate}
\end{corollary}
\begin{corollary}\label{3.9}
Let $m:X\times X\longrightarrow X$ be a bounded bilinear map and let $f:X\times X\times X\longrightarrow X$ be a bounded tri-linear map. Then 
\begin{enumerate}
\item If $f$ is regular and  first right (or second left) strongly irregular then $m$ is Arens regular.
\item If $m$ is Arens regular and  right (or left) strongly irregular then $f$ is regular.
\end{enumerate}
\end{corollary}

\begin{example}\label{3.10}
Let $G$ be an infinite, compact Hausdorff group and let $1 < p < \infty$. We know from \cite[pp 54]{9}, that $L^{p}(G)* L^{1}(G)\subset L^{p}(G)$ where 
$$(k* g)(x)=\int_{G}k(y)g(y^{-1}x)dy,\ \ \ \ \ \ (x\in G, k\in L^{p}(G),g\in L^{1}(G)).$$
On the other hand, the Banach space $L^{p}(G)$ is reflexive, thus the bounded tri-linear mapping
$$f:L^{p}(G)\times L^{1}(G)\times L^{p}(G)\longrightarrow L^{p}(G)$$
defined by $f(k,g,h)=(k* g)* h$, is regular for every $k,h\in L^{p}(G)$ and $g\in L^{1}(G)$. Therefore $Z_{l}^{2}(f)=L^{p}(G)^{**}=L^{p}(G)$, thus $f$ is second left strongly irregular.
\end{example}
\begin{theorem}\label{3.11}
Let $A$ be a Banach algebra. Then
\begin{enumerate}
\item If $(\Omega_{1},X)$ is a left Banach $A-$module and $\Omega_{1}^{***}, \pi^{***}(A,A)$ are factors then, $Z_{l}^{1}(\Omega_{1})\subseteq Z_{l}(\pi)$.
\item If $(X,\Omega_{2})$ is a right Banach $A-$module and $\Omega_{2}^{r***r}, \pi^{***}(A,A)$ are factors then, $Z_{r}^{2}(\Omega_{2})\subseteq Z_{r}(\pi)$.
\end{enumerate}
\end{theorem}
\begin{proof}
We prove only (1), the other one uses the same argument. Let $a^{**}\in Z_{l}^{1}(\Omega_{1})$, we show that $a^{**}\in Z_{l}(\pi)$. Let $\{b_{\alpha}^{**}\}$ be a net in $A^{**}$  which converge to $b^{**}\in A^{**}$ in the $w^{*}-$topologies. We must show that $\pi^{***}(a^{**},b_{\alpha}^{**})$ converge to $\pi^{***}(a^{**},b^{**})$  in the $w^{*}-$topologies. 
Let $a^{*}\in A^{*}$, since $\Omega_{1}^{***}$ factors, so there exist $x^{*}\in X^{*}, x^{**}\in X^{**}$ and $c^{**}\in A^{**}$ such that $a^{*}=\Omega_{1}^{***}(c^{**},x^{**},x^{*})$. In the other hands $\pi^{***}(A,A)$ factors, thus there exist $c,d\in A$ 
such that $\pi^{***}(c,d)=c^{**}$. Because $a^{**}\in Z_{l}^{1}(\Omega_{1})$ thus $\Omega_{1}^{****}(a^{**},b_{\alpha}^{**},x^{**})$  converge to $\Omega_{1}^{****}(a^{**},b^{**},x^{**})$ in the $w^{*}-$topologies.

In partiqular $\Omega_{1}^{****}(a^{**},b_{\alpha}^{**},\Omega_{1}^{****}(c,d,x^{**}))$ converge to $\Omega_{1}^{****}(a^{**},b^{**},$\\
$\Omega_{1}^{****}(c,d,x^{**}))$  in the $w^{*}-$topologies. Now by Theorem \ref{2.4}, we have
 \begin{eqnarray*}
\lim\limits_{\alpha}\langle \pi^{***}(a^{**},b_{\alpha}^{**}),a^{*}\rangle &=&\lim\limits_{\alpha}\langle \pi^{***}(a^{**},b_{\alpha}^{**}),\Omega_{1}^{***}(c^{**},x^{**},x^{*})\rangle \\
&=&\lim\limits_{\alpha}\langle \pi^{***}(a^{**},b_{\alpha}^{**}),\Omega_{1}^{***}(\pi^{***}(c,d),x^{**},x^{*})\rangle \\
&=&\lim\limits_{\alpha}\langle a^{**},\pi^{**}(b_{\alpha}^{**},\Omega_{1}^{***}(\pi^{***}(c,d),x^{**},x^{*}))\rangle \\
&=&\lim\limits_{\alpha}\langle a^{**},\Omega_{1}^{***}(b_{\alpha}^{**},\Omega_{1}^{****}(c,d,x^{**}),x^{*})\rangle \\
&=&\lim\limits_{\alpha}\langle \Omega_{1}^{****}(a^{**},b_{\alpha}^{**},\Omega_{1}^{****}(c,d,x^{**}),x^{*}\rangle \\
&=&\langle \Omega_{1}^{****}(a^{**},b^{**},\Omega_{1}^{****}(c,d,x^{**}),x^{*}\rangle \\
&=&\langle a^{**},\Omega_{1}^{***}(b^{**},\Omega_{1}^{****}(c,d,x^{**}),x^{*})\rangle \\
&=&\langle a^{**},\pi^{**}(b^{**},\Omega_{1}^{***}(\pi^{***}(c,d),x^{**},x^{*}))\rangle \\
&=&\langle a^{**},\pi^{**}(b^{**},\Omega_{1}^{***}(c^{**},x^{**},x^{*}))\rangle \\
&=&\langle a^{**},\pi^{**}(b^{**},a^{*})\rangle \\
&=&\langle \pi^{***}(a^{**},b^{**}),a^{*}\rangle.
\end{eqnarray*}
Therefore $\pi^{***}(a^{**},b_{\alpha}^{**})$ converge to $\pi^{***}(a^{**},b^{**})$  in the $w^{*}-$topologies, as required.
\end{proof}

\begin{theorem}\label{4.1}
Let $A$ be a Banach algebra and  $\Omega:A\times A\times A\longrightarrow A$ be a bounded tri-linear mapping. Then for every $a\in A, a^{*}\in A^{*}$ and $a^{**}\in A^{**}$,
\begin{enumerate}
\item If $A$ has a bounded right approximate identity and bounded linear map $T:A^{*}\longrightarrow A^{*}$ given by $T(a^{*})=\pi^{**}(a^{**},a^{*})$ be weakly compactenss, then $\Omega$ is regular.
\item If $A$ has a bounded left approximate identity and bounded linear map $T:A\longrightarrow A^{*}$ given by $T(a)=\pi^{r*r*}(a^{**},a)$ be weakly compactenss, then $\Omega$ is regular.
\end{enumerate}
\end{theorem}
\begin{proof}
We only prove (1). Let $T$ be weakly compact, then $T^{**}(A^{***})\subseteq A^{*}$. In the other hand, a direct verification reveals that $T^{**}(A^{***})=\pi ^{*****}(A^{**},A^{***})$. Thus $\pi ^{*****}(A^{*},A^{***})\subseteq A^{*}$. Now let $a^{**},b^{**}\in A^{**}$ and $a^{***}\in A^{***}$ too $\{a_{\alpha} \}$ and $\{a_{\beta}^{*} \}$ are nets in $A$ and $A^{*}$ which converge to $a^{**}$ and $a^{***}$ in the $w^{*}-$topologies, respectively. Then we have
\begin{eqnarray*}
\langle \pi^{*r***r}(a^{***},a^{**}),b^{**}\rangle &=&\langle \pi^{*r***}(a^{**},a^{***}),b^{**}\rangle=\langle a^{**},\pi^{*r**}(a^{***},b^{**})\rangle \\
&=&\lim\limits_{\alpha}\langle \pi^{*r**}(a^{***},b^{**}),a_{\alpha}\rangle =\lim\limits_{\alpha}\langle a^{***},\pi^{*r*}(b^{**},a_{\alpha})\rangle \\
&=&\lim\limits_{\alpha}\lim\limits_{\beta}\langle \pi^{*r*}(b^{**},a_{\alpha}),a_{\beta}^{*}\rangle=\lim\limits_{\alpha}\lim\limits_{\beta}\langle b^{**},\pi^{*r}(a_{\alpha},a_{\beta}^{*}) \rangle \\
&=&\lim\limits_{\alpha}\lim\limits_{\beta}\langle b^{**},\pi^{*}(a_{\beta}^{*},a_{\alpha})\rangle=\lim\limits_{\alpha}\lim\limits_{\beta}\langle \pi^{**}(b^{**},a_{\beta}^{*}),a_{\alpha}\rangle \\
&=&\lim\limits_{\alpha}\lim\limits_{\beta}\langle \pi^{***}(a_{\alpha},b^{**}),a_{\beta}^{*}\rangle =\lim\limits_{\alpha} \langle a^{***},\pi^{***}(a_{\alpha},b^{**})\rangle \\
&=&\lim\limits_{\alpha} \langle \pi^{****}(a^{***},a_{\alpha}),b^{**}\rangle =\lim\limits_{\alpha} \langle\pi^{*****}(b^{**},a^{***}),a_{\alpha}\rangle \\
&=&\langle a^{**},\pi^{*****}(b^{**},a^{***}) \rangle =\langle \pi^{****}(a^{***},a^{**}),b^{**}\rangle.
\end{eqnarray*}
Therefore $\pi^{*}$ is Arens regular. Now we implies that $A$ is reflexive by \cite[Theorem 2.1]{8}. It follows that $\Omega$ is regular and this completes the proof.
\end{proof}

\section{\textbf{Factors of bounded tri-linear mapping}}

We commence with the following definition.
\begin{definition}\label{5.1}
Let $X, Y, Z, S_{1}, S_{2}$ and $S_{3}$ be normed spaces,  $f:X\times Y\times Z\longrightarrow W$ and  $g:S_{1}\times S_{2}\times S_{3}\longrightarrow W$  be bounded tri-linear mappings. Then we say that $f$ factors through $g$ by bounded linear mappings $h_{1}:X\longrightarrow S_{1}$,  $h_{2}:Y\longrightarrow S_{2}$ and  $h_{3}:Z\longrightarrow S_{3}$, if $f(x,y,z)=g(h_{1}(x),h_{2}(y),h_{3}(z))$.
\end{definition}

In the next result we provide a necessary and sufficient condition such that first and second extension of a bounded tri-linear mapping $f$ factors through first and second extension of a bouneded tri-linear mapping $g$, respectively.
\begin{theorem}\label{5.2}
Let $f:X\times Y\times Z\longrightarrow W$ and $g:S_{1}\times S_{2}\times S_{3}\longrightarrow W$ be bounded tri-linear mapping.  Then 
\begin{enumerate}
\item The map $f$ factors through $g$ if and only if $f^{****}$ factors through $g^{****}$,
\item The map $f$ factors through $g$ if and only if $f^{r****r}$ factors through $g^{r****r}$.
\end{enumerate}
\end{theorem}
\begin{proof}
(1) Let $f$ factors through $g$ by  bounded linear mappings $h_{1}:X\longrightarrow S_{1}, h_{2}:Y\longrightarrow S_{2}$ and $h_{3}:Z\longrightarrow S_{3}$, then $f(x,y,z)=g(h_{1}(x),h_{2}(y),h_{3}(z))$ for every $x\in X, y\in Y$ and $z\in Z$. Let $\{x_{\alpha} \}, \{y_{\beta} \}$ and $\{z_{\gamma} \}$ be nets in $X, Y$ and $Z$  which converge to $x^{**}\in X^{**},y^{**}\in Y^{**}$ and $z^{**}\in Z^{**}$  in the $w^{*}-$topologies, respectively. Then for every $w^{*}\in W^{*}$ we have
\begin{eqnarray*}
\langle f^{****}(x^{**},y^{**},z^{**}),w^{*} \rangle&=&\lim\limits_{\alpha} \lim\limits_{\beta} \lim\limits_{\gamma}\langle w^{*},f(x_{\alpha},y_{\beta},z_{\gamma})\rangle \\
&=&\lim\limits_{\alpha} \lim\limits_{\beta} \lim\limits_{\gamma}\langle w^{*},g(h_{1}(x_{\alpha}),h_{2}(y_{\beta}),h_{3}(z_{\gamma}))\rangle\\
&=&\lim\limits_{\alpha} \lim\limits_{\beta} \lim\limits_{\gamma}\langle g^{*}(w^{*},h_{1}(x_{\alpha}),h_{2}(y_{\beta})),h_{3}(z_{\gamma})\rangle \\
&=&\lim\limits_{\alpha} \lim\limits_{\beta} \lim\limits_{\gamma}\langle h_{3}^{*}(g^{*}(w^{*},h_{1}(x_{\alpha}),h_{2}(y_{\beta}))),z_{\gamma}\rangle\\
&=&\lim\limits_{\alpha} \lim\limits_{\beta} \langle z^{**},h_{3}^{*}(g^{*}(w^{*},h_{1}(x_{\alpha}),h_{2}(y_{\beta})))\rangle \\
&=&\lim\limits_{\alpha} \lim\limits_{\beta} \langle h_{3}^{**}(z^{**}),g^{*}(w^{*},h_{1}(x_{\alpha}),h_{2}(y_{\beta}))\rangle\\
&=&\lim\limits_{\alpha} \lim\limits_{\beta} \langle g^{**}(h_{3}^{**}(z^{**}),w^{*},h_{1}(x_{\alpha})),h_{2}(y_{\beta})\rangle \\
&=&\lim\limits_{\alpha} \lim\limits_{\beta} \langle h_{2}^{*}(g^{**}(h_{3}^{**}(z^{**}),w^{*},h_{1}(x_{\alpha}))),y_{\beta}\rangle\\
&=&\lim\limits_{\alpha} \langle y^{**},h_{2}^{*}(g^{**}(h_{3}^{**}(z^{**}),w^{*},h_{1}(x_{\alpha})))\rangle \\
&=&\lim\limits_{\alpha} \langle h_{2}^{**}(y^{**}),g^{**}(h_{3}^{**}(z^{**}),w^{*},h_{1}(x_{\alpha}))\rangle \\
&=&\lim\limits_{\alpha} \langle g^{***}(h_{2}^{**}(y^{**}),h_{3}^{**}(z^{**}),w^{*}),h_{1}(x_{\alpha})\rangle \\
&=&\lim\limits_{\alpha} \langle h_{1}^{*}(g^{***}(h_{2}^{**}(y^{**}),h_{3}^{**}(z^{**}),w^{*})),x_{\alpha}\rangle\\
&=&\langle x^{**},h_{1}^{*}(g^{***}(h_{2}^{**}(y^{**}),h_{3}^{**}(z^{**}),w^{*}))\rangle \\
&=&\langle h_{1}^{**}(x^{**}),g^{***}(h_{2}^{**}(y^{**}),h_{3}^{**}(z^{**}),w^{*})\rangle\\
&=&\langle g^{****}( h_{1}^{**}(x^{**}),h_{2}^{**}(y^{**}),h_{3}^{**}(z^{**})),w^{*}\rangle.
\end{eqnarray*}
Therefore $f^{****}$ factors through $g^{****}$.

Conversely, suppose that $f^{****}$ factors through $g^{****}$, thus 
$$f^{****}(x^{**},y^{**},z^{**})= g^{****}( h_{1}^{**}(x^{**}),h_{2}^{**}(y^{**}),h_{3}^{**}(z^{**})), $$
in particular, for $x\in X, y\in Y$ and $z\in Z$ we have 
$$f^{****}(x,y,z)= g^{****}( h_{1}^{**}(x),h_{2}^{**}(y),h_{3}^{**}(z)).$$
Then for every $w^{*}\in W^{*}$ we have
\begin{eqnarray*}
\langle w^{*}&,&f(x,y,z)\rangle =\langle f^{*}(w^{*},x,y),z\rangle =\langle f^{**}(z,w^{*},x),y\rangle \\
&=&\langle f^{***}(y,z,w^{*}),x\rangle=\langle f^{****}(x,y,z),w^{*}\rangle  \\
&=&\langle g^{****}( h_{1}^{**}(x),h_{2}^{**}(y),h_{3}^{**}(z)),w^{*}\rangle =\langle h_{1}^{**}(x),g^{***}(h_{2}^{**}(y),h_{3}^{**}(z),w^{*})\rangle  \\
&=&\langle x,h_{1}^{*}(g^{***}(h_{2}^{**}(y),h_{3}^{**}(z),w^{*}))\rangle =\langle g^{***}(h_{2}^{**}(y),h_{3}^{**}(z),w^{*}),h_{1}(x)\rangle \\
&=&\langle h_{2}^{**}(y),g^{**}(h_{3}^{**}(z),w^{*},h_{1}(x))\rangle =\langle y,h_{2}^{*}(g^{**}(h_{3}^{**}(z),w^{*},h_{1}(x)))\rangle\\
&=&\langle g^{**}(h_{3}^{**}(z),w^{*},h_{1}(x)),h_{2}(y)\rangle =\langle h_{3}^{**}(z),g^{*}(w^{*},h_{1}(x),h_{2}(y))\rangle\\
&=&\langle z,h_{3}^{*}(g^{*}(w^{*},h_{1}(x),h_{2}(y)))\rangle =\langle g^{*}(w^{*},h_{1}(x),h_{2}(y)),h_{3}(z)\rangle\\
&=&\langle w^{*},g(h_{1}(x),h_{2}(y)),h_{3}(z))\rangle.
\end{eqnarray*}
It follows that $f$ factors through $g$ and this completes the proof.

(2) The proof  similar to (1).
\end{proof}

\begin{corollary}\label{5.4}
Let $f:X\times Y\times Z\longrightarrow W$ and $g:S_{1}\times S_{2}\times S_{3}\longrightarrow W$ be bounded tri-linear map and let $f$ factors through $g$. If $g$ is regular then $f$ is also regular.
\end{corollary}
\begin{proof}
Let $g$ be regular then $g^{****}=g^{r****r}$. Since the $f$ factors through $g$ then  for every $x^{**}\in X^{**}, y^{**}\in Y^{**}$ and $z^{**}\in Z^{**}$ we have
\begin{eqnarray*}
f^{****}(x^{**},y^{**},z^{**})&=&g^{****}( h_{1}^{**}(x^{**}),h_{2}^{**}(y^{**}),h_{3}^{**}(z^{**}))\\
&=&g^{r****r}( h_{1}^{**}(x^{**}),h_{2}^{**}(y^{**}),h_{3}^{**}(z^{**}))\\
&=&f^{r****r}(x^{**},y^{**},z^{**}).
\end{eqnarray*}
Therefore $f^{****}=f^{r****r}$, as claimed.
\end{proof}

\section{\textbf{Approximate identity and Factorization properties}}

Let $X$ be a Banach space, $A$ and $B$ be  Banach algebras  with  bounded left approximate identitis $\{e_{\alpha}\}$ and $\{e_{\beta}\}$, respactively.  Then a bounded tri-linear mapping $K_{1}:A\times B\times X\longrightarrow X$  is said to be left approximately unital if $$w^{*}-\lim\limits_{\beta}w^{*}-\lim\limits_{\alpha} K_{1}(e_{\alpha},e_{\beta},x)=x,$$
and  $K_{1}$ is said left unital if there exists $e_{1}\in A$ and $e_{2}\in B$ such that $K_{1}(e_{1},e_{2},x)=x$,  for every $x\in X$.

Similarly, bounded tri-linear mapping $K_{2}:X\times B\times A\longrightarrow X$  is said to be right approximately unital if $$w^{*}-\lim\limits_{\beta}w^{*}-\lim\limits_{\alpha} K_{1}(x,e_{\beta},e_{\alpha})=x,$$
and $K_{2}$ is also said to be right unital if $K_{2}(x,e_{2},e_{1})=x$.

\begin{lemma}\label{6.2}
Let $X$ be a Banach space, $A$ and $B$ be Banach algebras. Then bounded tri-linear mapping
\begin{enumerate}
\item $K_{1}:A\times B\times X\longrightarrow X$ is left approximately unital if and only if $K_{1}^{r****r}:A^{**}\times B^{**}\times X^{**}\longrightarrow X^{**}$ is left unital.
\item $K_{2}:X\times B\times A\longrightarrow X$ is right approximately unital if and only if $K_{2}^{****}:X^{**}\times B^{**}\times A^{**}\longrightarrow X^{**}$ is right unital.
\end{enumerate}
\end{lemma}
\begin{proof}
We prove only (1), the other part have the same argument. Let $K_{1}$ be a left approximately unital. Thus  there exists bounded left approximate identitys $\{e_{\alpha}\}\subseteq A$  and $\{e_{\beta}\}\subseteq B$ shuch that $$w^{*}-\lim\limits_{\beta}w^{*}-\lim\limits_{\alpha} K_{1}(e_{\alpha},e_{\beta},x)=x,$$
for every $x\in X$. Let $\{e_{\alpha} \}$ and $\{e_{\beta} \}$ which converge to $e_{1}^{**}\in A^{**}$ and $e_{2}^{**}\in B^{**}$  in the $w^{*}-$topologies, respectively. In the other hand, for every $x^{**}\in X^{**}$, let $\{x_{\gamma} \}\subseteq X$ which converge to $x^{**}$ in the $w^{*}-$topologies, then we have
\begin{eqnarray*}
\langle K_{1}^{r****r}(e_{1}^{**},e_{2}^{**},x^{**}),x^{*}\rangle &=&\langle K_{1}^{r****}(x^{**},e_{2}^{**},e_{1}^{**}),x^{*}\rangle\\
&=&\langle x^{**}, K_{1}^{r***}(e_{2}^{**},e_{1}^{**},x^{*})\rangle\\
&=&\lim\limits_{\gamma}\langle K_{1}^{r***}(e_{2}^{**},e_{1}^{**},x^{*}),x_{\gamma}\rangle \\
&=&\lim\limits_{\gamma}\langle e_{2}^{**},K_{1}^{r**}(e_{1}^{**},x^{*},x_{\gamma})\rangle\\
& =&\lim\limits_{\gamma}\lim\limits_{\beta}\langle K_{1}^{r**}(e_{1}^{**},x^{*},x_{\gamma}),e_{\beta}\rangle \\
&=&\lim\limits_{\gamma}\lim\limits_{\beta}\langle e_{1}^{**},K_{1}^{r*}(x^{*},x_{\gamma},e_{\beta})\rangle\\
& =&\lim\limits_{\gamma}\lim\limits_{\beta}\lim\limits_{\alpha}\langle K_{1}^{r*}(x^{*},x_{\gamma},e_{\beta}),e_{\alpha}\rangle \\
&=&\lim\limits_{\gamma}\lim\limits_{\beta}\lim\limits_{\alpha}\langle x^{*},K_{1}^{r}(x_{\gamma},e_{\beta},e_{\alpha})\rangle \\
& =&\lim\limits_{\gamma}\lim\limits_{\beta}\lim\limits_{\alpha}\langle x^{*},K_{1}(e_{\alpha},e_{\beta},x_{\gamma})\rangle \\
&=&\lim\limits_{\gamma}\langle x^{*},x_{\gamma}\rangle=\langle x^{**},x^{*}\rangle.
\end{eqnarray*}
Therefore $K_{1}^{r****r}(e_{1}^{**},e_{2}^{**},x^{**})=x^{**}$. It follows that $K_{1}^{r****r}$ is left unital.

Conversely, suppose that $K_{1}^{r****r}$ is left unital. So there exists $e_{1}^{**}\in A^{**}$ and $e_{2}^{**}\in b^{**}$ shuch that $K_{1}^{r****r}(e_{1}^{**},e_{2}^{**},x^{**})=x^{**}$ for every $x^{**}\in X^{**}$. Now let $\{e_{\alpha} \}, \{e_{\beta} \}$ and $\{x_{\gamma} \}$ be nets in $A, B$ and $X$  which converge to $e_{1}^{**}, e_{2}^{**}$ and $x^{**}$  in the $w^{*}-$topologies, respectively. Thus 
\begin{eqnarray*}
w^{*}-\lim\limits_{\gamma}w^{*}-\lim\limits_{\beta}w^{*}-\lim\limits_{\alpha} K_{1}(e_{\alpha},e_{\beta},x_{\gamma})&=&K_{1}^{r****r}(e_{1}^{**}, e_{2}^{**},x^{**}) \\
&=&x^{**}=w^{*}-\lim\limits_{\gamma} x_{\gamma}.
\end{eqnarray*}
 Therefore $K_{1}$ is left approximately unital and this completes the proof.
\end{proof}
\begin{remark}\label{6.3}
It should be remarked that in contrast to the situation occurring for $K_{1}^{r****r}$ and $K_{2}^{****}$ in the above lemma, $K_{1}^{****}$ and $K_{2}^{r****r}$ are not necessarily left and right unital respectively, in general.
\end{remark}
\begin{theorem}\label{6.4}
Suppos $X, S$ are Banach spaces  and $A, B$ are Banach algebras.
\begin{enumerate}
\item Let $K_{1}:A\times B\times X\longrightarrow X$ be left approximately unital and factors through $ g_{r}:A\times B\times S\longrightarrow X$ from rigth by $h:X\longrightarrow S$. If $h$ is weakly compactenss, then $X$ is reflexive.
\item Let $K_{2}:X\times B\times A\longrightarrow X$ be right approximately unital and factors through $ g_{l}:S\times B\times A\longrightarrow X$ from left by $h:X\longrightarrow S$. If $h$ is weakly compactenss, then $X$ is reflexive.
\end{enumerate}
\end{theorem}
\begin{proof}
We only give a proof for (1). Since $K_{1}$ is left approximately unital, thus there exists $e_{1}^{**}\in A^{**}$ and $e_{2}^{**}\in B^{**}$ shut that 
$$K_{1}^{r****r}(e_{1}^{**},e_{2}^{**},x^{**})=x^{**},$$
for every $x^{**}\in X^{**}$. In the other hand, the bounded tri-linear mapping $K_{1}$ factors through $ g_{r}$ from rigth, so Theorem \ref{5.2} follows that $K_{1}^{r****r}$ factors through $g_{r}^{r****r}$ from rigth. In other words 
$$K_{1}^{r****r}(e_{1}^{**},e_{2}^{**},x^{**})=g_{r}^{r****r}(e_{1}^{**},e_{2}^{**},h_{3}^{**}(x^{**})).$$ Then for every $x^{***}\in X^{***}$ we have
\begin{eqnarray*}
\langle x^{***},x^{**}\rangle &=&\langle x^{***},K_{1}^{r****r}(e_{1}^{**},e_{2}^{**},x^{**})\rangle\\
&=&\langle x^{***},g_{r}^{r****r}(e_{1}^{**},e_{2}^{**},h^{**}(x^{**}))\rangle\\
&=&\langle g_{r}^{r****r*}(x^{***},e_{1}^{**},e_{2}^{**}),h^{**}(x^{**})\rangle \\
&=&\langle h^{***}(g_{r}^{r****r*}(x^{***},e_{1}^{**},e_{2}^{**})),x^{**}\rangle.
\end{eqnarray*}
Therefore $x^{***}=h^{***}(g_{r}^{r****r*}(x^{***},e_{1}^{**},e_{2}^{**}))$. The weak compactness of $h$ implies that of $h^{*}$ from which we have $h^{***}(S^{***})\subseteq X^{*}$. In particular 
$$h^{***}(g_{r}^{r****r*}(x^{***},e_{1}^{**},e_{2}^{**}))\subseteq X^{*},$$
that is, $X^{*}$ is reflexive. So $X$ is reflexive.
\end{proof}




\end{document}